\documentclass[11pt]{amsart}
\usepackage[english]{babel}
\usepackage{amsmath,amssymb,amsthm, amscd}
\usepackage[alphabetic]{amsrefs}
\usepackage{enumitem}
\usepackage[unicode,pdfencoding=auto,psdextra]{hyperref}
\usepackage{xcolor}  %%  to highlight comments during the writing-up phase
\usepackage{mathrsfs}
\usepackage{graphicx}
\usepackage{booktabs}
\usepackage[T1]{fontenc}

\newtheorem{theorem}{Theorem}[section]

\newtheorem{lemma}[theorem]{Lemma}
\newtheorem{proposition}[theorem]{Proposition}
\theoremstyle{definition}
\newtheorem{definition}[theorem]{Definition}

\theoremstyle{remark}
\newtheorem{remark}[theorem]{Remark}
\newtheorem{problem}[theorem]{Problem}
\numberwithin{equation}{section}

\newcommand{\C}{\mathbb{C}} % the complex numbers
\newcommand{\R}{\mathbb{R}} % the real numbers
\newcommand{\Z}{\mathbb{Z}}
\newcommand{\N}{\mathbb{N}}

\newcommand{\holo}{\mathcal{O}}
\DeclareMathOperator{\tr}{tr}
\DeclareMathOperator{\lie}{Lie}
\DeclareMathOperator{\aut}{Aut}
\DeclareMathOperator{\vf}{VF}

\DeclareMathOperator{\diag}{diag}

\title[Non-Runge Fatou-Bieberbach Domains]%
{Non-Runge Fatou-Bieberbach Domains in Stein Manifolds with the Density Property and Non-tame Families of Stein Structures}
\author{Gaofeng Huang \and Frank Kutzschebauch \and Feng Rong}

\address{Department of Mathematics, Harvard University, Cambridge, MA 02138, USA}

\email{ghuang@math.harvard.edu}

\address{Mathematisches Institut, Universit\"at Bern, Sidlerstrasse 5, CH--3012 Bern, Switzerland}

\email{frank.kutzschebauch@unibe.ch}

\address{School of Mathematical Sciences, Shanghai Jiao Tong University, 800 Dong Chuan Road, Shanghai 200240, P.R. China}

\email{frong@sjtu.edu.cn}

\thanks{The first author was funded by Schweizerischer Nationalfonds [207335] and  Postdoc.Mobility [P500-$2\_239113$]. The second author was partially supported by Schweizerischer Nationalfonds [207335] and [10005432]. The third author was partially supported by the National Natural Science Foundation of China (Grant No. 12671101).}

\subjclass[2020]{Primary 32M17; Secondary 32E30}

\keywords{Fatou-Bieberbach domain, Anders{\'e}n-Lempert theory, density property, tame families of Stein structures.}

\begin{document}

\begin{abstract}
    Let $X$ be a Stein manifold with the density property. We present methods of constructing two kinds of non-Runge Fatou-Bieberbach domains in $X$, which by definition are proper open subsets of $X$ biholomorphic either to $\mathbb{C}^n$ with $n = \dim X$ or to $X$. For both kinds, we provide examples where our methods apply. The results are applied to give new examples of non-tame families of Stein structures on manifolds with non-trivial topology.
\end{abstract}

\maketitle

\section{Introduction}

By Picard's Little Theorem, a non-constant holomorphic function $f: \C \to \C$ can avoid at most one value, in which case $f$ is an exponential function plus a constant and thus never injective. The situation is completely different in several complex variables. Indeed,  injective and non-surjective holomorphic maps $F: \C^n \to \C^n$ ($n \ge 2$) do exist. These maps are called Fatou-Bieberbach maps and their images $\Omega = F(\C^n) \subsetneq \C^n$ are called {\bf Fatou-Bieberbach domains}. The existence of these maps is called the  Poincar{\'e}-Fatou-Bieberbach phenomenon  and has been observed independently by Fatou, Bieberbach and Poincar{\'e}. 

There are two known methods to produce such maps and domains in $\C^n$. The first method was presented with a complete proof in Rosay-Rudin \cite{MR0929658}. For an automorphism of $\C^n$ with an attracting fixed point, its basin of attraction is biholomorphic to $\C^n$. In particular, if some points do not belong to the image (e.g.\@ there exists another fixed point), then the basin of attraction is a Fatou-Bieberbach domain. The second method is the so-called push-out method attributed to Dixon and Esterle \cite{MR0854551}. It is worth mentioning that a long-standing open problem in functional analysis, Michael's conjecture, was in that paper reduced to a question about the existence of certain sequences of Fatou-Bieberbach domains. 

In both methods, the Fatou-Bieberbach map is constructed as a limit of holomorphic automorphisms of $\C^n$. In other words, the constructions use the fact that $\C^n$ ($n \ge 2$) has an enormously big group of holomorphic automorphisms. Generalizing from $\C^n$, an interesting class of Stein manifolds with a big automorphism group has been introduced by Varolin \cite{MR1829353}, namely Stein manifolds with the density property (see Definition \ref{def: DP}). In \cite{MR1785520}, Varolin observed that the Poincar{\'e}-Fatou-Bieberbach phenomenon exists for any  Stein manifold $X$ with the density property in two ways similar to the case of $\C^n$. The first method of basin of attraction produces biholomorphic images of $\C^n$, $n= \dim_\C X$, in $X$, which we call {\bf Fatou-Bieberbach domains of the first kind}. The second method, the push-out method,  produces biholomorphic images of $X$ in itself different from $X$, which we call {\bf Fatou-Bieberbach domains of the second kind}. 

By their very construction, being domains of convergence of a sequence of holomorphic automorphisms of $X$, both methods produce Runge domains of $X$. Recall that a domain $\Omega \subseteq X$ is said to be {\bf Runge} if any holomorphic function on $\Omega$ can be approximated uniformly on compact subsets in $\Omega$ by holomorphic functions on $X$. The long-standing question whether there exist non-Runge Fatou-Bieberbach domains in $\C^n$ was settled by Wold in the positive (\cite{MR2372737}). The aim of this paper is to study the existence of non-Runge Fatou-Bieberbach domains in Stein manifolds with the density property. 

Our main results, Theorems \ref{thm: typeI} and \ref{thm: typeII}, give sufficient conditions for the existence of non-Runge Fatou-Bieberbach domains of the first kind and the second kind, respectively. For both kinds, we provide examples where the conditions of our theorems are satisfied (see Theorems \ref{example first kind}, \ref{KR}, \ref{example second kind}, \ref{Hc} and \ref{cubic}). The examples of non-Runge Fatou-Bieberbach domains of the second kind seem to be much more difficult to find than those of the first kind (see the discussion before Problem \ref{SL2_H}). In a forthcoming paper, we will provide more examples of non-Runge Fatou-Bieberbach domains of both kinds. 

The paper is organized as follows. In Section \ref{prelim}, we recall the Anders\'en-Lempert theorem and its relative version, as well as criteria for the density property. We also introduce a new notion of weak density property. Next, in Sections \ref{sec: typeI} and \ref{sec: typeII}, we prove Theorems \ref{thm: typeI} and \ref{thm: typeII}. In Section \ref{sec: example}, we apply these theorems to find non-trivial examples of non-Runge Fatou-Bieberbach domains of both kinds. Finally in Section \ref{sec: wild}, we apply the non-Runge Fatou-Bieberbach domains of the second kind to construct the first examples of non-tame families of Stein structures on $SU_n \times \R^{n^2-1} \cong SL_n (\C)$ for every $n\ge 2$. These are new examples of such structures on manifolds with non-trivial topology.

\section{Preliminaries} \label{prelim}

In this section, we recall some tools from the Anders\'en-Lempert theory. For more information, see the monograph by Forstneri{\v c} \cite{MR2975791}*{\S 4} or the recent survey by Forstneri{\v c}-Kutzschebauch \cite{MR4440754}.  

\begin{definition}[Varolin \cite{MR1829353}] \label{def: DP}
    A complex manifold $X$ has the {\bf density property} if the Lie algebra generated by $\C$-complete holomorphic vector fields on $X$ is dense in the Lie algebra of all holomorphic vector fields on $X$, with respect to the compact-open topology.
\end{definition}

The following Anders{\'e}n-Lempert theorem is the main implication of the density property (see Anders{\'e}n-Lempert \cite{MR1185588}, Forstneri\v{c}-Rosay \cite{MR1213106}, Varolin \cite{MR1829353}).
\begin{theorem} \label{andlemp}
Let $X$ be a Stein manifold with the density property. If $\Phi_t : \Omega \stackrel{\cong}{\to} \Omega_t=\Phi_t(\Omega)\subseteq X$ $(t\in [0,1])$ is a continuous isotopy of biholomorphic maps between Stein Runge domains in $X$ where $\Phi_0: \Omega \to X$ is the inclusion map,  then $\Phi_1$ can be approximated uniformly on compact subsets in $\Omega$ by holomorphic automorphisms of $X$.
\end{theorem}

Denote by $\mathrm{Aut}(X)$ the group of holomorphic automorphisms of $X$ and by $\holo(X)$ the ring of holomorphic functions on $X$. For any $x \in X$, a finite subset $S$ of the tangent space $T_{x}X$ is called a \textbf{generating set} if the image of $S$ under the action induced by the isotropy subgroup at $x$ (in $\mathrm{Aut}(X)$) spans $T_{x}X$ as a complex vector space. A pair $(V,W)$ of complete holomorphic vector fields on $X$ is a \textbf{compatible pair} if the closure $\overline{ \mathrm{Span}} (\ker V \cdot \ker W)$ (in the compact-open topology) contains a non-trivial ideal $I \subseteq \holo(X)$ and there exists a holomorphic function $h \in \ker V$ such that $W(h) \in \ker W \backslash \{0\}$. 

The following is a criterion by Kaliman and the second author \cite{MR2385667}*{Theorem 2} (see also Leuenberger \cite{MR3513546}). 
\begin{theorem} \label{theorem:KalKutHolomorph}
    Let $X$ be a Stein manifold on which $\mathrm{Aut}(X)$ acts transitively. Assume that there are compatible pairs $(V_1,W_1), (V_2, W_2), \dots, (V_m, W_m)$ of complete holomorphic vector fields and a point $x \in X$ such that the vectors $V_1(x), \dots, V_m(x)$ form a generating set for $T_x X$. Then $X$ has the density property.  
\end{theorem}

\subsection{Weak density property}
Let $X$ be a Stein space, and $Y\subseteq X$ a closed analytic subvariety 
containing the set $X_{sing}$ of singularities of $X$. Let $\holo_X$ be the sheaf of holomorphic functions on $X$ and $\mathcal{T}_X$ the tangent sheaf of $X$. Let $I_Y \subseteq \holo(X)$ be the ideal of $Y$. We write $\vf_{hol}(X)$ for the $\holo(X)$-module of holomorphic vector fields on $X$. Let $\vf_{hol}(X,Y)$ be the $\holo(X)$-module of holomorphic vector fields vanishing on $Y$:
\[	
	\vf_{hol}(X,Y)= \{V \in \vf_{hol}(X) : V(\holo(X)) \subseteq I_Y\}, 
\]
and $\vf^Y_{hol}(X)$ the $\holo(X)$-module of of holomorphic vector fields on $X$ tangent to $Y$:
\[
    \vf^Y_{hol}(X)= \{V \in \vf_{hol}(X) : V(y) \in T_y Y \, \, \forall y \in Y \}.
\]
Denote by $\lie_{hol}(X,Y)$ the Lie algebra generated by complete vector fields in $\vf_{hol}(X,Y)$ and by $\aut^Y(X)$ the group of holomorphic automorphisms of $X$ preserving $Y$
\[
    \aut^Y(X) = \{ \alpha \in \aut(X) : \alpha(Y) =  Y \}. 
\]

\begin{definition}[Kutzschebauch-Leuenberger-Liendo \cite{MR3320241}] \label{def:relDP} 
We say that $X$ has the {\bf density property relative to $Y$}, or $(X,Y)$ has the relative density property, if there exists an integer $\ell\geq 0$ such that $I_Y^\ell \vf_{hol}(X,Y)$ is contained in the closure of $\lie_{hol}(X,Y)$ in the compact-open topology.
\end{definition}

 We introduce the following version of density property. (See \cite{MR1829353} for a relevant discussion in the $\C^n$ setting.)

\begin{definition} \label{weakDP}
    We say that $X$ has the {\bf weak density property tangent to} $Y$, or $(X,Y)$ has the weak density property, if there exists a coherent sheaf $\mathcal{W}$ of $\holo_X$-submodules of $\vf^Y_{hol}(X)$ such that
    \begin{enumerate}
        \item every global section of $\mathcal{W}$ is contained in the closure of the Lie algebra generated by complete vector fields in $\vf^Y_{hol}(X)$, and 
        \item for every $x \in X \backslash Y$, the stalk of $\mathcal{W}$ at $x$ coincides with that of the tangent sheaf $\mathcal{T}_X$.
    \end{enumerate}
\end{definition}
Clearly, the relative density property of $(X, Y)$ implies the weak density property. 

\begin{theorem} \label{weakDP-criterion}
    Let $X$ be a Stein space, and $Y\subseteq X$ a closed analytic subvariety containing the set $X_{sing}$ of singularities of $X$. Suppose that
    \begin{enumerate}
        \item $\aut^Y(X)$ acts transitively on $X \backslash Y$, and 
        \item there are compatible pairs $(V_1,W_1), (V_2, W_2), \dots, (V_m, W_m)$ of complete vector fields in $\vf_{hol}^Y(X)$, and
        \item there exists a point $x \in X \backslash Y$ such that the vectors $V_1(x), \dots, V_m(x)$ form a generating set (w.r.t.\@ $\aut^Y(X)$ instead of $\aut(X)$) for $T_x X$.
    \end{enumerate}
    Then $X$ has the weak density property tangent to $Y$.  
\end{theorem}
\begin{proof}
    Together with these three assumptions, the first part of the proof of \cite{MR3320241}*{Theorem 2.2} (replace algebraic objects by holomorphic ones) produces a coherent sheaf of $\holo_X$-submodules of $\vf^Y_{hol}(X)$ which satisfies the two conditions in Definition \ref{weakDP}. 
\end{proof}

The main purpose of introducing the weak density property is the following version of Theorem \ref{andlemp} with interpolation, which generalizes the relative Anders{\'e}n-Lempert theorem \cite{MR3320241}*{Theorem 6.3}.

\begin{theorem} \label{weakAL-Theorem} %
  Let $X$ be a Stein space with the weak density property tangent to a
  closed analytic subvariety $Y$ containing $X_{sing}$. Assume that $\Omega \subseteq X$ is a Stein Runge domain and $ \Phi_t :\Omega \to
  X \,(t \in [0,1])$ is a continuous isotopy of injective holomorphic maps as in Theorem \ref{andlemp}. Moreover, if $\Phi_t(\Omega) \cap Y = \emptyset$ for every $t\in[0,1]$, then $\Phi_1$ can be approximated uniformly on compact subsets in $\Omega$ by holomorphic automorphisms of $X$ which preserve $Y$. 
\end{theorem}
\begin{proof}
    The idea is the same as for Theorem \ref{andlemp} (see e.g.\@ \cite{MR3023850}*{Appendix}). Here we indicate where the weak density property is needed. By definition, there exists a coherent sheaf $\mathcal{W}$ of $\holo_X$-submodules of $\vf_{hol}^Y(X)$ as in Definition \ref{weakDP}. 
    
    Taking derivative of the isotopy $\Phi_t$, we obtain time dependent vector fields $V_t$ on $\Phi_t(\Omega)$ for $t \in [0,1]$. After dividing the interval $[0,1]$ into small enough subintervals $[t_k, t_{k+1}]$ and choosing the time independent vector fields $V_{t_k}$ on $\Phi_{t_k}(\Omega)$, one needs to approximate the vector field $V_{t_k}$ defined on $\Phi_{t_k}(\Omega)$ by a global section of $\mathcal{W}$ on a compact subset $K \subseteq \Phi_{t_k}(\Omega)$. 

    Without loss of generality, assume that $K$ is holomorphically convex in $\Phi_{t_k}(\Omega)$. Let $U$ be a Stein open neighborhood of $K$ such that $K \subseteq U \subseteq \Phi_{t_k}(\Omega)$. By Cartan's Theorem A, there exist finitely many global sections $s_1, s_2, \dots, s_N$ of $\mathcal{W}$ which generate the stalk of $\mathcal{W}$ at every point of $U$.
    
    Since $\mathcal{W}$ and $\mathcal{T}_X$ are identical on the level of stalks on $X \backslash Y$, and we can also regard $s_1, s_2, \dots, s_N$ as global sections of $\mathcal{T}_X$, we have that the morphism of coherent sheaves 
    \[
        \varphi : \holo_X^N \rvert_U \to \mathcal{T}_X \rvert_U, \, \, (f_1, f_2, \dots, f_N) \mapsto f_1 s_1 + f_2 s_2 + \cdots + f_N s_N
    \]
    is surjective on the level of stalks. This induces the short exact sequence
    \[
        0 \to \ker \varphi \to \holo_X^N \rvert_U \to \mathcal{T}_X \rvert_U \to 0,
    \]
    from which we deduce the long exact sequence
    \[
        \cdots \to H^0(U, \holo_X^N \rvert_U) \to H^0(U, \mathcal{T}_X \rvert_U) \to H^1(U, \ker \varphi) \to \cdots.
    \]
    By the coherence of $\ker \varphi$ and Cartan's Theorem B, we have $H^1(U, \ker \varphi) = 0$. Thus, the map $H^0(U, \holo_X^N \rvert_U) \to H^0(U, \mathcal{T}_X \rvert_U)$ is surjective and every section of $\mathcal{T}_X \rvert_U$ is of the form $\sum_i f_i s_i$ with $f_i \in \holo(U)$. Hence, $V_{t_k} \rvert_U$ is equal to $\sum_i f_i s_i$ for some $f_i \in \holo(U)$. Moreover, since $K$ is holomorphically convex in $\Phi_{t_k}(\Omega)$ and $\Phi_{t_k}(\Omega)$ is Runge in $X$, there exists $g_i \in \holo(X)$ which approximates $f_i$ on $K$. Thus the vector field $\sum_i g_i s_i$, which is a global section of $\mathcal{W}$, approximates $V_{t_k}$ on $K$.  

    Furthermore, since global sections of $\mathcal{W}$ can be uniformly approximated by Lie combinations of complete vector fields in $\vf^Y_{hol}(X)$, the flow of $V_{t_k}$ for small enough time can be uniformly approximated by composition of  holomorphic automorphisms in $\aut^Y(X)$. 
\end{proof}

\begin{remark} \label{emptyInt}

    (1) Compared to the relative density property of $(X, Y)$, the advantage of the weak version is that we can use any vector fields on $X$ which are tangent to $Y$, while still have the freedom of moving compact subsets in $X \backslash Y$ around.  Depending on which complete vector fields from $\vf^Y_{hol}(X)$ have been used in the criteria in Theorem \ref{weakDP-criterion}, we get different Lie subalgebras of $\vf^Y_{hol}(X)$ satisfying Varolin's original density property. For example, when only complete vector fields vanishing on $Y$ are used, we get the relative density property. 

    (2) If $(X, Y)$ has the weak density property, then so does $(X, \alpha(Y))$ for any $\alpha \in \aut(X)$, since completeness and Lie brackets of vector fields are preserved by automorphisms.  
\end{remark}

\section{Fatou-Bieberbach domains of the first kind} \label{sec: typeI}

In \cite{MR2372737}, Wold modified Stolzenberg's example \cite{MR0203080} to construct a pair of disjoint totally real disks whose union $Y$ is holomorphically convex in $\C^* \times \C$, while its polynomially convex hull $\widehat{Y}$ contains the origin of $\C^2$. By the density property of $\C^* \times \C$ (\cite{MR1829353}), there exist a Fatou-Bieberbach domain $\Omega$ of the first kind and a holomorphic automorphism $\Psi$ of $\C^* \times \C$ such that $\Psi(Y) \subseteq \Omega$. Thus $\Psi^{-1}(\Omega)$ is a Fatou-Bieberbach domain which contains $Y$ but not $\widehat{Y}$. We generalize Wold's method to Stein manifolds with the density property. 

\begin{theorem}\label{thm: typeI}
    Let $X$ be a Stein manifold and $H \subseteq X$ a closed complex hypersurface such that $X \backslash H$ has the density property. Then, there exists a Fatou-Bieberbach domain of the first kind in $X \backslash H$, which is Runge in $X \backslash H$ but non-Runge in $X$.
\end{theorem}
\begin{proof}
The complement of a closed hypersurface of pure codimension one in a Stein manifold is a Stein manifold \cite{MR2029201}*{p.~129}.

\noindent {\bf Step 1}.  
By the density property of $X \backslash H$, there is a Fatou-Bieberbach domain $\Omega \cong \C^n$ in $X \backslash H$, where $n = \dim_\C X$. This can be obtained as the basin of attraction of an attracting fixed point of an automorphism of $X \backslash H$, see \cite{MR1785520}*{Corollary 4.1}. It then also follows that $\Omega$ is Runge in $X \backslash H$. 
\medskip

\noindent
{\bf Step 2}. 
Fix $p_0\in H_{reg}$. Let $U$ be a Runge open neighborhood of $p_0$ in $X$ and find an injective holomorphic map $\varphi: U\rightarrow \mathbb{C}^n$ such that $\varphi(H \cap U)=\{z_1=0\} \cap \varphi(U)$ and $\varphi(p_0)=0$. Fix a sufficiently small $\epsilon > 0$. Consider a small polydisk $\Delta^n_\epsilon \subseteq V:=\varphi(U)$ centered at 0. Set $D:=\varphi^{-1}(\Delta^n_\epsilon)$. Note that the polydisk $\Delta^n_\epsilon$ is Runge in $V$, thus the preimage $D$ is Runge in $U = \varphi^{-1}(V)$.

Consider the embedding
$$
    \iota: \C^2 \hookrightarrow \C^n, (x, y) \mapsto (\epsilon x/R, \epsilon y/R, 0, \dots, 0).
$$ 
Set $Z = \iota(Y)$, with $Y \subseteq \Delta^2_R$ ($R >0$) given in \cite{MR2372737} as above. Then $Z = Z_1 \cup Z_2 \subseteq \Delta^n_\epsilon \backslash \{z_1 =0 \}$ is a disjoint union of two totally real disks, such that $Z$ is holomorphically convex in $\Delta^n_\epsilon \backslash \{z_1 =0\}$ and the polynomially convex hull of $Z$ intersects the hyperplane $\{z_1=0\}$.

Now we set $W:= \varphi^{-1}(Z) = W_1 \cup W_2 \subseteq D \backslash H$. For any domain $M \supseteq W$, denote by $\widehat{W}_M$ the $\holo(M)$-convex hull of $W$. Note that $\widehat{W}_D=\widehat{W}_X$ and
\begin{align} \label{WhullH}
    \widehat{W}_D \cap H = \widehat{W}_X \cap H  \neq \emptyset.
\end{align}
Meanwhile, $W$ is holomorphically convex in $D \backslash H$. 
By the next step, $W$ is holomorphically convex in $X \backslash H$ as well. 
\medskip

\noindent
{\bf Step 3}. We show that $D\backslash H$ is Runge in $X\backslash H$.
Let $K$ be a compact subset of $D\backslash H$. 
Since $D$ is Stein and Runge in $X$, 
\[
    \widehat{K}_X = \widehat{K}_D \subseteq D. 
\]
Then $\widehat{K}_{X \backslash H}$ lies in $\widehat{K}_X \cap (X \backslash H) \subseteq D \backslash H$. Since $X \backslash H$ is Stein, the hull $\widehat{K}_{X \backslash H}$ of the compact $K$ is compact in $X \backslash H$. 

Let $f$ be a holomorphic function on $D\backslash H$. The Oka-Weil theorem applied to $\widehat{K}_{X \backslash H} \subseteq X \backslash H$ yields a holomorphic function on $X \backslash H$ approximating $f$ uniformly on $K$.

\medskip

\noindent
{\bf Step 4}. Here we use the density property of $X \backslash H$ to move $W = W_1 \cup W_2$ into the Fatou-Bieberbach domain $\Omega$. First, there exists a continuous isotopy of local diffeomorphisms which map the two totally real disks $W_i, i = 1,2$ into separate sufficiently small balls $B_i$ (which are preimages of small balls in $\Delta_\epsilon^n$). Then, we approximate this isotopy of local diffeomorphisms by a continuous isotopy of injective holomorphic maps of $D \backslash H$ using \cite{MR1314745}*{Corollary 3.2}, which in turn is approximable by a continuous isotopy  of holomorphic automorphisms of $X \backslash H$ by Theorem \ref{andlemp}.

Since $X \backslash H$ is connected, we can move $B_1 \cup B_2$ by a continuous isotopy of local biholomorphic maps into $\Omega$, while keeping the images of $B_1 \cup B_2$ holomorphically convex in $X \backslash H$. Finally, \cite{MR1213106}*{Lemma 2.2} and Theorem~\ref{andlemp} give a continuous isotopy of holomorphic automorphisms of $X \backslash H$ which approximates this isotopy of local maps. 

By these two isotopies of holomorphic automorphisms of $X \backslash H$, we have an automorphism $\Phi$ taking $W \subseteq X \backslash H$ into $\Omega$. 
Then, the preimage $\Phi^{-1}(\Omega)$ is a Fatou-Bieberbach domain of the first kind containing $W$, which is holomorphically convex in $X \backslash H$. 

It remains to see that $\Phi^{-1}(\Omega)$ is non-Runge in $X$. If $\Phi^{-1}(\Omega)$ were Runge in $X$, then $\widehat{W}_X=\widehat{W}_{\Phi^{-1}(\Omega)}$. Since $\Phi^{-1}(\Omega)$ is Stein, $\widehat{W}_X \subseteq \Phi^{-1}(\Omega) \subseteq X \backslash H$. However, this contradicts \eqref{WhullH}.
\end{proof}

\section{Fatou-Bieberbach domains of the second kind} \label{sec: typeII}

In this section, we construct non-Runge Fatou-Bieberbach domains of the second kind using similar ideas as in the last section, together with the push-out method.

Let $X$ be a Stein manifold and $H \subseteq X$ a complex hypersurface. We say that the pair $(X,H)$ satisfies {\bf Property (PO)} if there exists an isotopy $\alpha_t \in \aut(X)$ $(t \in [0, \infty))$ $C^1$-jointly in $(t,x)$ with $\alpha_0 =Id$ such that for any compact subset $K \subseteq X$  there is a $s_K \ge 0$ with $\alpha_t(H) \cap K = \emptyset$ for all $t \ge s_K$.

\begin{theorem} \label{thm: typeII}
    Let $X$ be a Stein manifold and $H \subseteq X$ a closed complex hypersurface.  Assume that
    \begin{enumerate}
        \item $X \backslash H$ has the density property, and
        \item $X$ has the weak density property tangent to $H$, and
        \item $(X,H)$ satisfies Property (PO) with corresponding isotopy $\alpha_t $.
    \end{enumerate}
    Then, there exists a Fatou-Bieberbach domain $\Omega$ of the second kind in $X$ which is contained in $X \backslash H$, Runge in $X \backslash H$, but non-Runge in $X$.

    Additionally, let $\emptyset \neq K_0 \Subset L$ be holomorphically convex compact subsets with  $L \cap \alpha_t (H) =\emptyset$ for all $t\ge 0$, and let $\epsilon >0$. Then $\Omega$ can be chosen so that $K_0 \subseteq \Omega$,  $\widehat{\varphi (L)}$ contains a point in $H$ and 
    $$ \Vert \varphi^{-1}- Id \Vert_{K_0} < \epsilon,$$
    where $\varphi: X \to \Omega$ is the biholomorphic Fatou-Bieberbach map.
\end{theorem}
\begin{proof} 
    We prove the existence of $\Omega$ with the additional assumptions. The first claim follows from this once we know that such a pair $K_0, L$ always exists. Let $E := \cup_{t \ge 0} \alpha_t(H)$. By Property (PO), for every compact $K \subseteq X$ we have 
    $$
        E \cap K = K \cap \bigcup_{ t \in [0, s_K]} \alpha_t(H)
    $$ 
    which is closed and thus $E$ is closed in $X$. Since $(t,x) \mapsto \alpha_t(x)$ is $C^1$ and $[0, \infty) \times H$ has real dimension $2n-1$, $E$ has measure zero. Pick $p \in X \backslash E$ and a proper holomorphic embedding $X \hookrightarrow \C^N$. For small $r >0$, the sets $K_0 := X \cap \overline{B}_p(r/2) \Subset L := X \cap \overline{B}_p(r)$ lie in $X \backslash E$ and are holomorphically convex. 
    
    %We now prove the second assertion.
    Let $K_1 \Subset K_2 \Subset \cdots $ be an exhaustion of $X$ by holomorphically convex compact subsets, with $K_0 \Subset L \Subset K_1$ and remember  $L \cap H=\emptyset$. Choose an increasing sequence $\{s_i\}_i$ of real numbers such that $\alpha_{s}(H) \cap K_i = \emptyset$ for $s \ge s_i$. Set $H_i = \alpha_{s_i}(H)$.  

    By assumption $\alpha_s(H) \cap L = \emptyset$ for all $s \ge 0$, thus we can take $s_0 = 0$ and $\alpha_{s_0}= Id$. %Suppose that $H \cap K_1\neq \emptyset$.

    Let 
    \begin{align} \label{gamma-t}
         \gamma_t = \alpha_{s_1} \circ \alpha_{\sigma(t)}^{-1} \,\, \text{ with } \,\,  \sigma(t) = (1-t) s_1 + t s_{0}.
    \end{align}
    Then  we have 
    \[
        \gamma_t(L ) \cap H_1 = \emptyset \,\, \text{ for } t \in [0,1].  
    \]
    By assumption (2) and Remark \ref{emptyInt} (2), $X$ has the weak density property tangent to $H_1$. 
    Let $U$ be a Stein open neighborhood of $L$ which is Runge in $X$ and $ \gamma_t (U) \cap H_1 = \emptyset$ for all $t \in [0,1]$.   Thus, the additional condition in Theorem \ref{weakAL-Theorem} is satisfied.
    Applying Theorem \ref{weakAL-Theorem} to $Y = H_1$ and the isotopy $\gamma_t \rvert_U : U \to X $ $(t \in [0,1])$ we obtain an automorphism $\beta_1$ preserving $H_1$ such that the composition
    \[
        \theta_1 = \beta_1^{-1} \circ \alpha_{s_1} \circ \alpha_{s_{0}}^{-1} 
    \]
    takes $H_0=H$ to $H_1$ and is close to the identity on $L$. 

    Inductively, we apply Theorem \ref{weakAL-Theorem} to $Y = H_i$, a Stein Runge neighborhood of $K_{i-1}$ disjoint from the union 
    $$
        \bigcup_{s \in [s_{i-1},s_i]} \alpha_s(H)
    $$ 
    and an isotopy as \eqref{gamma-t} to get an automorphism $\beta_i$ preserving $H_i$ such that 
    \[
        \theta_i = \beta_i^{-1} \circ \alpha_{s_i} \circ \alpha_{s_{i-1}}^{-1}
    \]
    takes $H_{i-1}$ to $H_i$ and is close to the identity on $K_{i-1}$. %(without loss of generality, we assume that $H_{i-1} \cap K_i\neq \emptyset$ for each $i\ge 1$). 
    The sequence of 
    \[
        \Psi_m := \theta_m \circ \theta_{m-1} \circ \cdots \circ \theta_1,\ \ \ \ \ \ m\ge 1,
    \]
    converges uniformly on compact subsets in the domain 
    $$\Omega' := \bigcup_{m=1}^\infty \Psi_m^{-1} (K_{m-1})$$ 
    to a limit $\Psi$, which maps $\Omega'$ biholomorphically onto $X$ (see \cite{MR1760722}*{Proposition 5.1}). By our construction, $\Omega'$ is Runge in $X$, and $\Omega' \cap H = \emptyset$ since $\Psi_m(H) = H_m$ does not intersect $K_m$.
    %the complement of $\Omega' \subsetneq X$ contains a copy of $H$. 
    In addition  $\lim_{m\rightarrow \infty} \Psi_m$ is arbitrarily close to the identity on $L$ if every $\theta_i$ is chosen sufficiently close to the identity on $K_{i-1}$.

    Finally, following Steps 2-4 as in the proof of Theorem \ref{thm: typeI}, we get an automorphism $\Phi$ such that $\Phi^{-1}(\Omega')$ is a Fatou-Bieberbach domain of the second kind in $X$ which is contained in $X \backslash H$, Runge in $X \backslash H$, but non-Runge in $X$.

    In Step 4 where we move the Stolzenberg disks in $X \backslash H $ by the automorphism $\Phi \in \aut(X \backslash H)$ into $\Omega'$, we move them into $\Psi^{-1}(L) \backslash K_0$, which by construction is contained in the domain of convergence $\Omega'$ of the above sequence of automorphisms.  At this stage we can assume that on a slightly smaller holomorphically convex compact $L_1 \Subset L $ containing $K_0$ in its interior, the automorphism $\Phi$ is arbitrarily close to the identity. For this, just observe that the union of $L_1$ and the images of the two disjoint balls $B_1$ and $B_2$ can be chosen holomorphically convex (if the balls are small enough and their paths do stay away from $L_1$). Then in the application of the Anders{\'e}n-Lempert theorem the continuous isotopy of local biholomorphisms can be prescribed as the identity on $L_1$.

    We set 
    $$
        \varphi = \Phi^{-1} \circ \Psi^{-1} :  X \to \Phi^{-1}(\Omega')=:\Omega. 
    $$
    Since $\lim_{m\rightarrow \infty} \Psi_m$ is arbitrarily close to the identity on $L\Supset K_0$ and $\Phi$ is arbitrarily close to the identity on $L_1 \Supset K_0$, we conclude firstly  that $L$ is in the domain of convergence of the above sequence of automorphisms and we secondly get the approximation part of the last assertion.  Since  the Stolzenberg disks were moved into $\Psi^{-1}(\mathrm{int}\, L) \backslash L_1$,  $\widehat{ \varphi (L)}$ contains a point in $H$. The last assertion of the theorem is proven.   

    Since $\Phi(K_0) \subseteq \mathrm{int}\, L \subseteq \Omega'$, it follows $K_0 \subseteq \Phi^{-1}(\Omega') = \Omega$. 
\end{proof}

\section{Examples} \label{sec: example}

\subsection{The first kind} 

We apply Theorem \ref{thm: typeI} to produce non-Runge Fatou-Bieberbach domains of the first kind in $SL_n(\C), \, n \ge 2$, and the Koras-Russell cubic threefold. 

\begin{lemma} \label{DP-SLnH}
    Let $H = \{A \in SL_n(\C):\ \tr A =0\}$. Then, $SL_n(\C) \backslash H$ has the density property.
\end{lemma}
\begin{proof}
Let $X = SL_n(\C)$ and $A = (a_{ij})_{i,j=1}^n\in X$. Set $h(A) =\tr A$. We verify the conditions of Theorem \ref{theorem:KalKutHolomorph}.
\smallskip

(a) At the identity $I_n \in X$, since the adjoint representation of $X$ is irreducible, the span of any orbit of the adjoint action  in $T_{I_n} X \backslash \{0 \}$ is $T_{I_n} X $. Note that any inner automorphism of $X$ preserves the hypersurface $H$. For any complete vector field $V$ non-vanishing at $I_n$, since its orbit under the adjoint action contains a set of complete vector fields spanning the tangent space, $V(I_n)$ is a generating set for $T_{I_n}(X \backslash H)$. 

\smallskip

(b) The same idea using the adjoint representation shows that $\mathrm{Aut}(X \backslash H)$ acts transitively on $X \backslash H$. Consider the complete vector field 
\begin{align*}
    V_{12}(A) = h(A) E_{21} A = \left(\sum_k a_{kk} \right) \left( a_{11} \frac{\partial}{\partial a_{21}} + a_{12} \frac{\partial}{\partial a_{22}} + \dots + a_{1n} \frac{\partial}{\partial a_{2n}}   \right). 
\end{align*}
The conjugation 
$$\gamma_g : X \to X, \quad A \mapsto g A g^{-1} $$
restricts to an automorphism of $X \backslash H$ by conjugation invariance of the trace. As in the case with the adjoint representation, for $g \in SL_n(\C)$ the complete vector fields  
$$(\gamma_g)_* V_{12} (A) = h(A) g E_{21} g^{-1} A$$
span the tangent space at any $A \in X \backslash H$. Thus every $\mathrm{Aut}(X \backslash H)$ orbit is open and since $X \backslash H$ is connected, the transitivity follows.

\smallskip

(c) We have the following pair of complete vector fields on $X$ which are tangent to $H$: $V = V_{12}$ and
\begin{align*}
        W = a_{21} \frac{\partial}{\partial a_{11}}+(a_{22} - a_{11})\frac{\partial}{\partial a_{12}} - a_{21} \frac{\partial}{\partial a_{22}}
\end{align*}
for $n = 2$ and
\begin{align*}
    W = & \,\, a_{23} \left( a_{21} \frac{\partial}{\partial a_{11}} + a_{22} \frac{\partial}{\partial a_{12}} + \dots + a_{2n} \frac{\partial}{\partial a_{1n}}   \right) \\
    &- a_{21} \left( a_{21} \frac{\partial}{\partial a_{31}} + a_{22} \frac{\partial}{\partial a_{32}} + \dots + a_{2n} \frac{\partial}{\partial a_{3n}}   \right).
\end{align*}
for $n \ge 3$. 
The function $a_{11}$ satisfies  $a_{11} \in \ker V$ and $W(a_{11}) \in \ker W \backslash \{0 \}$. Note that $h = \sum_i a_{ii}  \in \ker W$, i.e.\@ the flow of $W$ preserves $\{h = \text{const}\}$, thus $h^k \in \ker W$ for $k \in \Z$. For $n \ge 3$, we have $\C[a_{21},\dots,a_{2n}, h^{-1}] \subseteq \ker W$, hence $\textup{Span}(\ker V \cdot \ker W) \supseteq \C[X \backslash H] \cong \C[X][h^{-1}]$. As for $n=2$, one can get the same inclusion from $a_{11}, a_{12} \in \ker V$ and $a_{21}, a_{11} + a_{22} \in \ker W$. Taking the closure over the compact-open topology, we get 
    \[
        \overline{\mathrm{Span}}(\ker V \cdot \ker W) = \holo( X \backslash H).
    \] 
    This gives a compatible pair for $X \backslash H$.  
\end{proof}

Combining Theorem \ref{thm: typeI} and Lemma \ref{DP-SLnH}, we get the following.

\begin{theorem} \label{example first kind}
    Let $H=\{ A \in SL_n(\C):\ \tr A=0\}$. Then, there exists a Fatou-Bieberbach domain of the first kind in $SL_n(\C) \backslash H$, which is Runge in $SL_n(\C) \backslash H$ but non-Runge in $SL_n(\C)$. 
\end{theorem}

% Testing water
Our next example is the Koras-Russell cubic threefold $KR$, a smooth complex hypersurface in $\C^4$ given by $\{ x^2 y + x + s^2 + t^3 = 0 \}$. It is known that $KR$ has the density property by Leuenberger \cite{MR3513546}. (It is open whether $KR$ is biholomorphic to $\C^3$.)

The hypersurface $H := \{x = 0\} \cap KR$ is isomorphic to 
$$
    \C_y \times \{(s,t) \in \C^2: s^2 + t^3 = 0 \},
$$ 
which is singular along $\C_y \times \{(0,0)\}$.
%Note that $x V_x, xV_x', V_z$ induce complete vector fields on $KR \backslash H$.

\begin{lemma} \label{lem: KR-DP}
    $KR \backslash H$ has the density property. 
\end{lemma}
\begin{proof}
    On the complement of $H=\{x=0\}$ we substitute $y$ and get the isomorphism
    \[
        KR \backslash H \to \C^* \times \C^2, \quad  (x,y,s,t) \mapsto (x, s, t).
    \]
    The density property of $\C^* \times \C^2$ is known \cite{MR1829353}. 
\end{proof}

Combining Theorem \ref{thm: typeI} and Lemma \ref{lem: KR-DP}, we get the following, cf.\@ Wold's $\C^* \times \C \subseteq \C^2$. 
\begin{theorem}\label{KR}
    Let $H=\{ (x,y,s,t) \in KR:\ x=0 \}$. Then, there exists a Fatou-Bieberbach domain of the first kind in $KR \backslash H$, which is Runge in $KR \backslash H$ but non-Runge in $KR$. 
\end{theorem}

\subsection{The second kind}

A smooth affine algebraic variety $X$ has the {\bf algebraic density property} if the Lie algebra generated by $\C$-complete algebraic vector fields on $X$ contains all algebraic vector fields on $X$. 

The algebraic density property of $X$ implies the density property for $X$, since holomorphic vector fields on $X$ can be approximated by polynomial vector fields on $X$ uniformly on compact subsets of $X$. The same holds for the algebraic density properties introduced below (cf.\@ \cite{MR3320241}*{Proposition~6.2}).

Let $X$ be an affine algebraic variety, and $ Y \subseteq X$ an algebraic subvariety containing the singularity part $X_{sing}$. Denote by $I_Y\subseteq \C[X]$ the ideal of $Y$, and by $\vf_{alg}(X)$ the $\C[X]$-module of algebraic vector fields on $X$.
Let $\vf_{alg}(X,Y)$ be the $\C[X]$-module of vector fields vanishing on $Y$:
\[	
	\vf_{alg}(X,Y)=\{V \in \vf_{alg}(X) : V(\C[X])\subseteq I_Y\},
\]
and $\vf^Y_{alg}(X)$ the $\C[X]$-module of of algebraic vector fields on $X$ tangent to $Y$:
\[
    \vf^Y_{alg}(X)= \{V \in \vf_{alg}(X) : V(y) \in T_y Y \, \, \forall y \in Y \}.
\]

Denote by $\lie_{alg}(X,Y)$ the Lie algebra generated by complete vector fields in $\vf_{alg}(X,Y)$ and by $\aut^Y_{alg}(X)$ the group of algebraic automorphisms of $X$ preserving $Y$
\[
    \aut^Y_{alg}(X) = \{ \alpha \in \aut_{alg}(X) : \alpha(Y) = Y \}. 
\]

\begin{definition}[Kutzschebauch-Leuenberger-Liendo \cite{MR3320241}] \label{def:relADP} 
We say that $X$ has the {\bf algebraic density property relative to $Y$}, or $(X,Y)$ has the relative algebraic density property, if there exists an integer $\ell\geq 0$ such that $I_Y^\ell \vf_{alg}(X,Y) \subseteq \lie_{alg}(X,Y)$.
\end{definition}

We also introduce the algebraic counterpart of the weak density property. 
\begin{definition} \label{weakADP}
    We say that $X$ has the {\bf weak algebraic density property tangent to} $Y$, or $(X,Y)$ has the weak algebraic density property, if there exists a coherent sheaf $\mathcal{W}$ of $\C[X]$-submodules of $\vf^Y_{alg}(X)$ such that
    \begin{enumerate}
        \item every global section of $\mathcal{W}$ is contained in the Lie algebra generated by complete vector fields in $\vf^Y_{alg}(X)$, and 
        \item for every $x \in X \backslash Y$, the stalk of $\mathcal{W}$ at $x$ coincides with that of $\mathcal{T}_X$.
    \end{enumerate}
\end{definition}
As in the holomorphic case, the relative algebraic density property of $(X, Y)$ implies the weak algebraic density property.

\begin{theorem} \label{weakADP-criterion}
    Let $X$ be an affine algebraic variety, and $Y\subseteq X$ an algebraic subvariety containing the set $X_{sing}$ of singularities of $X$. Suppose that
    \begin{enumerate}
        \item $\aut^Y_{alg}(X)$ acts transitively on $X \backslash Y$, and 
        \item there are compatible pairs $(V_1,W_1), (V_2, W_2), \dots, (V_m, W_m)$ of complete vector fields in $\vf_{alg}^Y(X)$, and
        \item there exists a point $x \in X \backslash Y$ such that the vectors $V_1(x), \dots, V_m(x)$ form a generating set (w.r.t.\@ $\aut^Y_{alg}(X)$) for $T_x X$.
    \end{enumerate}
    Then $X$ has the weak algebraic density property tangent to $Y$.  
\end{theorem}
\begin{proof}
    The first part of the proof of \cite{MR3320241}*{Theorem 2.2} produces the desired coherent sheaf of $\C[X]$-submodules of $\vf^Y_{alg}(X)$. 
\end{proof}
\medskip 

We recall the notion of flexibility introduced in \cite{MR3039680}. 
Let $X$ be a smooth affine algebraic variety. 
A derivation $V$ on the ring of regular functions $\C[X]$ is called locally nilpotent (an LND), if for each $f \in \C[X]$ there is a natural number $n$ such that $V^n f = 0$.
An algebraic vector field on $X$ is locally nilpotent, if it as a derivation on $\C[X]$ is locally nilpotent. 
We call $X$ \textbf{flexible}, if locally nilpotent vector fields on $X$ span the tangent space $T_x X$ at every point $x \in X$.

The following proposition is of independent interest (cf.\@ \cite{nTuple}*{Theorem 3} and \cite{MR1829353}).

\begin{proposition} \label{lem: ADP-XxC*}
    Let $X$ be a smooth flexible affine algebraic variety of dimension $n \ge 1$. Then 
    \begin{enumerate}
        \item $X \times \C$ and $X \times \C^*$ have the algebraic density property.
        \item $X \times \C$ has the algebraic density property relative to $X \times \{0\}$. 
    \end{enumerate}
\end{proposition}
\begin{proof}
    (1) These are special cases of \cite{MR2385667}*{Corollary 2.8}.

    \smallskip

    (2) Denote by $z$ the coordinate on $\C$. 
    For any nonzero LND $V$ on $X$, consider the pair of complete vector fields $(zV, z \partial / \partial z)$ on $X \times \C$. Since $z^k \in \ker zV$ for any $k \in \N$ and $\C [X] \subseteq \ker z\partial/\partial z$, we have 
    $$
        \C[X \times \C] \subseteq \mathrm{Span} \left(\ker zV \cdot \ker z\frac{\partial}{ \partial z} \right).
    $$ 
    Thus we have a compatible pair $(zV, z \frac{\partial}{\partial z})$ for $X \times \C$. More precisely, the $\C[X \times \C]$-submodule generated by this pair is 
    \begin{align*}
        \C[X \times \C] V(h) z^2 \frac{\partial}{\partial z},
    \end{align*}
    where $h \in \C[X]$ satisfies $V(h) \in \ker V\backslash \{0\}$. Choose $x_0 \in X$ so that $V(h)(x_0) \neq 0$. %At $(x_0,1)$, the fiber of this submodule contains the generating set $z \partial / \partial z$. 
    
    Let $V_1, V_2, \dots, V_n$ be LNDs spanning the tangent space $T_{x_0} X$. For each $V_i$, the time-$1$ flow of $z (z-1)V_i$ fixes $(x_0,1)$ and its differential at this point is an endomorphism of the tangent space: 
    \begin{align*}
        z \frac{\partial }{\partial z} \Big\rvert_{z=1} \mapsto z \frac{\partial }{\partial z}\Big\rvert_{z=1} +  V_i(x_0), \quad i = 1, \dots, n.
    \end{align*}
    Together, they form a basis of the tangent space $T_{(x_0,1)}(X \times \C)$, and thus $z \partial / \partial z$ is a generating set at $(x_0, 1)$.

    Moreover, the group $\mathrm{Aut}_{alg}(X \times \C, X \times \{0\})$ of algebraic automorphisms of $X \times \C$ fixing $X \times \{0\}$ acts transitively on $X \times \C^*$, since for a fixed $z \in \C^*$, the transitivity along the fiber follows from the flexibility of $X$ and along the $z$ direction, the complex flow of $z \partial / \partial z$ has orbit $\C^*$.

    The claim then follows from \cite{MR3320241}*{Theorem 2.2}. 
\end{proof}

Combining Theorem \ref{thm: typeII} and Proposition \ref{lem: ADP-XxC*}, we get the following.

\begin{theorem} \label{example second kind}
    Let $X$ be a smooth flexible affine algebraic variety. Then, there exists a Fatou-Bieberbach domain of the second kind in $X \times \C$ which is contained in $X \times \C^*$, Runge in $X \times \C^*$ but non-Runge in $X \times \C$.
\end{theorem}
\begin{proof}
By Proposition \ref{lem: ADP-XxC*}, the complement $X \times \C^*$ of $X \times \{0\}$ in $X\times \C$ has the density property and $X \times \C$ has the density property relative to $X \times \{0\}$. Taking translation along $\C$, the last condition of Theorem \ref{thm: typeII} is also satisfied.
\end{proof}

\begin{remark} \label{final}
    (1) Since the conditions of Theorem \ref{thm: typeI} are part of those of Theorem \ref{thm: typeII}, there is a non-Runge Fatou-Bieberbach domain of the first kind in $X \times \C$ when $X$ is a smooth flexible affine algebraic variety. 

    (2) The Calogero-Moser space $CM$ is the product of $X \times \C^2$ for a smooth affine subvariety $X$, see Andrist-Huang \cite{MR4986786}. It can be shown that $X$ is flexible, thus by Theorem \ref{example second kind}, $CM$ contains a non-Runge Fatou-Bieberbach domain which is biholomorphic to $CM$. 
\end{remark}

Our next example is again about $SL_n(\C)$, $n \ge 2$, but with different complex hyperplanes.

\begin{theorem}\label{Hc}
    Let $H_c =\{ A =(a_{ij})_{i,j=1}^n \in SL_n(\C):\ a_{11} = c \neq 0 \}$. Then, there exists a Fatou-Bieberbach domain of the second kind in $SL_n(\C)$ contained in $SL_n(\C) \backslash H_c$, which is Runge in $SL_n(\C) \backslash H_c$ but non-Runge in $SL_n(\C)$. 
\end{theorem}
\begin{proof}
    (1) Let $X = SL_n(\C)$. We first show the density property of $X \backslash H_c$. Consider the following complete vector fields on $X$
    \begin{align*}
        V_{ij} &=  a_{i1} \frac{\partial}{\partial a_{j1}} + a_{i2} \frac{\partial}{\partial a_{j2}} + \dots + a_{in} \frac{\partial}{\partial a_{jn}}   , \\
        W_{ij} &=  a_{1i} \frac{\partial}{\partial a_{1j}} + a_{2i} \frac{\partial}{\partial a_{2j}} + \dots + a_{ni} \frac{\partial}{\partial a_{nj}}   , \quad 1 \le i \neq j \le n.
    \end{align*}
    The flow of $V_{ij}$ adds the $i$th row to the $j$th row, whereas the flow of $W_{ij}$ adds the $i$th column to the $j$th column. 
    
    We consider the induced complete vector fields $(a_{11}-c)V_{21}, W_{12}$ and $V_{12}$ on $X \backslash H_c$. To show the density property we invoke a variant of the holomorphic version of Theorem \ref{weakADP-criterion} in the sense that we use the so-called compatible triples, discussed in \cite{nTuple}, in place of compatible pairs to generate a submodule. 

    (i) Denote by $\lie_{H_c}$ the Lie algebra generated by complete holomorphic vector fields on $X \backslash H_c$. For $f \in \ker (a_{11}-c)V_{21} = \ker V_{21}$ and $g \in \ker W_{12}$
    \begin{align*}
        &\quad \,\, [ f(a_{11}-c)V_{21}, a_{22} g W_{12} ] - [ a_{22} f (a_{11}-c)V_{21}, g W_{12} ] \\
        &= fg W_{12}(a_{22}) (a_{11}-c) V_{21} = fg a_{21} (a_{11}-c) V_{21} \in \lie_{H_c} .
    \end{align*}
    Then for $h \in \ker V_{12}$ we have
    \begin{align*}
        &\quad \,\,[ h V_{12}, f (a_{11} g) a_{21} (a_{11}-c) V_{21} ] - [ a_{11} h V_{12}, f g a_{21} (a_{11}-c) V_{21} ] \\
        &= fgh a_{21} (a_{11}-c) V_{21}(a_{11}) V_{12}  = fgh (a_{11}-c) a_{21}^2 V_{12} \in \lie_{H_c} .
    \end{align*}
    Since $\C[X \backslash H_c] \subseteq \mathrm{Span}(\ker (a_{11}-c) V_{21} \cdot \ker W_{12} \cdot \ker V_{12})$ we obtain the desired submodule
    \begin{align*}
        (a_{11}-c) a_{21}^2 \,\C[X \backslash H_c]  \, V_{12} \subseteq \lie_{H_c} .
    \end{align*}

    (ii) We show the transitivity of $\aut(X \backslash H_c)$. 
    Let $A = (a_{ij})_{i,j=1}^n \in X \backslash H_c$. We may assume $a_{21} \neq 0$, for otherwise we could use the flow of $V_{j2}$ (when $a_{j1} \neq 0$) to change $a_{21}$. 
    Under the flow of $(a_{11}-c)V_{21}$ the variable $a_{21}$ remains constant and 
    \[
        a_{11} \mapsto c + (a_{11}-c) e^{a_{21} t},
    \]
    thus using this flow we move $a_{11}$ to $c' \notin \{ 0, c \}$. 

    Next, we use the flows of $V_{1j}, W_{1j}$, $j \ge 2$, to clear the first column and the first row. What remains is a block matrix $\diag (c', A')$, which can be sent to $\diag(c', 1/c', 1, \dots, 1)$ by multiplying with $\diag(1,g), g \in SL_{n-1}(\C)$.  

    \smallskip
    (iii) We find a generating set. Assume $c \neq 1$, otherwise a diagonal scaling moves $c$. Denote by $E_{ij}$ the matrix with $a_{ij}=1$ and all other entries zero. At $A_0 = I_n + E_{21}$ we have $a_{11}=a_{21}=1$. 

    For a complete vector field $V$ tangent to $H_c$ with $V(a_{21})= 0$, the time-$1$ flow of $(a_{21}-1)V$ fixes $A_0$ and its differential at $A_0$ sends
    \[
        V_{12}(A_0) \mapsto V_{12}(A_0) + a_{11}(A_0) V(A_0) = \frac{\partial}{\partial a_{21}} +  V(A_0).
    \]
    Applying this to the following vector fields on the first column of Table \ref{tab: sln-spanning}, we obtain a spanning set of vectors at $T_{A_0} (X\backslash H_c)$ which has $\mathfrak{sl}_n \cdot A_0$ as the left trivialization. Hence $\{ V_{12}(A_0)  \}$ is a generating set for $T_{A_0} (X\backslash H_c)$.

    \begin{table}[ht]
    \centering
    \renewcommand{\arraystretch}{1.25}
    \begin{tabular}{llll}
    \toprule
    $V$  & $B$ with $V(A_0)=B A_0$ & new element of $\mathfrak{sl}_n$\\
    \midrule
    $V_{12}$  & $E_{21}$ & $E_{21}$\\
    $(a_{11}-c)V_{21}$  & $(1-c)E_{12}$ & $E_{12}$\\
    $V_{ji}$, $i \ge 3$, $j \neq i$ & $E_{ij}$ & $E_{ij}$, $i \ge 3$\\
    $W_{2j}$, $j \ge 3$ & $E_{2j}$ & $E_{2j}$, $j\ge3$\\
    $W_{1j}$, $j \ge 3$ & $E_{1j}+E_{2j}$ & $E_{1j}$, $j\ge3$\\
    $W_{12}$ & $E_{12}-E_{11}+E_{22}-E_{21}$ & $E_{22}-E_{11}$\\
    $R_{E_{ii}-E_{jj}}$, $i,j\ge2$, $i \neq j$ & $E_{ii}-E_{jj}-(\delta_{i2}-\delta_{j2})E_{21}$ & $E_{ii}-E_{jj}$, $i,j \ge2$\\
    \bottomrule
    \end{tabular}
    \medskip
    \caption{The vectors obtained by applying the isotopy group action at $A_0=I_n+E_{21}$, $c\neq1$.}\label{tab: sln-spanning}
    \end{table}
    In Table \ref{tab: sln-spanning}, we write tangent vectors at $A_0$ as right trivialization $B A_0$ with $B \in \mathfrak{sl}_n$. In the last row, $R_D$ denotes the vector field $A\mapsto AD$ of $t \mapsto A\exp(tD)$. 
    For the vector field $W_{ij}$, the valuation at $A_0$ is $W_{ij}(A_0) = A_0 E_{ij}$. To express it with respect to the left trivialization at $A_0$ we conjugate them back. For instance for $W_{ij}$ one has 
    \[
        B = A_0 E_{ij}A_0^{-1}=E_{ij} + \delta_{1i} E_{2j}-\delta_{j2} E_{i1}-\delta_{1i}\delta_{j2}E_{21}. 
    \]
    %The last column contains all $E_{ij}$, $i\neq j$, and $E_{ii}-E_{jj}$.

    \smallskip
    (2) The same arguments in (1) show the weak density property of $X$ tangent to $H_c$. 

    \smallskip
    (3) Since $c$ is nonzero, Property (PO) for $(X, H_c)$ follows as multiplication by $\diag(e^s,e^{-s},1,\dots,1)$ for $s \ge 0$ maps $H_c$ onto $H_{e^s c}$ which leaves any compact subset $K \subseteq X=SL_n(\C)$ for large $s$. 
\end{proof}

Finally, we consider a family of smooth hypersurfaces in $\C^4$ given by $X_\beta = \{ x^2 y + x + s^2 + t^\beta = 0 \}$ with $\beta \ge 2$. It is known that $X_\beta$ has the density property by Leuenberger \cite{MR3513546}. For $\beta=3$ it is the Koras-Russell cubic threefold.

\begin{theorem}\label{cubic}
    Let $H_c =\{ (x,y,s,t) \in X_\beta:\ t = c \neq 0 \}$. Then, there exists a Fatou-Bieberbach domain of the second kind in $X_\beta $ contained in $X_\beta \backslash H_c$, which is Runge in $X_\beta \backslash H_c$ but non-Runge in $X_\beta$. 
\end{theorem}
\begin{proof}
    (1) We first show the density property of $X_\beta \backslash H_c$. The following vector fields are complete on $X_\beta$ (see \cite{MR3513546}*{Lemma 3.1}):
    \begin{align*}
        V_x &= - 2s \frac{\partial}{\partial y} + x^2 \frac{\partial}{\partial s}, \quad V_x' = - \beta t^{\beta-1} \frac{\partial}{\partial y} + x^2 \frac{\partial}{\partial t}, \\
        V_y &= -2s \frac{\partial}{\partial x} + (2xy + 1) \frac{\partial}{\partial s}. 
    \end{align*}
    We consider the induced complete vector fields $V_x, (t-c)V_x'$ and $V_y$ on $X_\beta \backslash H_c$. To show the density property, we again use compatible triples to generate a submodule. 

    (i) Denote by $\lie_{H_c}$ the Lie algebra generated by complete holomorphic vector fields on $X_\beta \backslash H_c$. For $f \in \ker (t-c) V_x' = \ker V_x'$ and $g \in \ker V_x$
    \begin{align*}
        &\quad \,\, [ f(t-c)V_x', sg V_x ] - [ sf(t-c)V_x', g V_x ] \\
        &= fg V_x(s) (t-c) V_x' = fg x^2 (t-c) V_x' \in \lie_{H_c} .
    \end{align*}
    Then for $h \in \ker V_y$ we have
    \begin{align*}
        &\quad \,\,[ h V_y, f(tg) x^2 (t-c) V_x' ] - [ t h V_y, f g x^2 (t-c) V_x' ] \\
        &= fgh x^2 (t-c) V_x'(t) V_y  = fgh (t-c) x^4 V_y \in \lie_{H_c} .
    \end{align*}
    Since $\C[X_\beta \backslash H_c] \subseteq \mathrm{Span}(\ker (t-c) V_x' \cdot \ker V_x \cdot \ker V_y)$ we obtain the desired submodule
    \begin{align*}
        (t-c) x^4 \,\C[X_\beta \backslash H_c]  \, V_y \subseteq \lie_{H_c} .
    \end{align*}

    (ii) We show the transitivity of $\aut(X_\beta\backslash H_c)$. Under the time-$\tau$ flow of $(t-c)V_x'$ the variables $x,s$ remain constant and 
    \[
        t \mapsto c + (t-c) e^{x^2 \tau}, \quad y \mapsto y - \frac{ t(\tau)^\beta - t^\beta}{x^2}.
    \]
    On the complement $X_\beta \backslash \{ x=0 \} \cong \C^*_x \times \C_{s,t}^2$ the variables $x,s,t$ are global coordinates and the projections of the above three complete fields are $x^2\partial / \partial s$, $(t-c)x^2\partial / \partial t$ and $-2s\partial / \partial x+(2xy+1)\partial / \partial s$. These three span the tangent space if $s x^4 (t-c) \neq 0$. 

    Let $U = \{x \neq 0 , s \neq 0 , t \neq c\} \cap X_\beta$.  For a point $p = (x,y,s, t) \in X_\beta \backslash H_c$, if $x \neq 0$ and $s= 0$, the flow of $V_x$ makes $s \neq 0$. If $x=0$, we use the flow of $V_y$ to bring $p$ into $U$.

    (iii) We find a generating set. Choose a point $p_0 = (x_0,y_0,s_0,t_0) \in X_\beta \backslash H_c$ with $x_0 s_0 \neq 0$. Then the time-$1$ flow of $(x-x_0)(t-c)V_x'$ fixes $p_0$ and its differential at $p_0$ sends 
    \[
        V_y(p_0) \mapsto V_y(p_0) - 2 s_0 (t_0-c) V_x'(p_0)
    \]
    while the time-$1$ flow of $(x-x_0)V_x$ fixes $p_0$ and its differential sends
    \[
        V_y(p_0) \mapsto V_y(p_0) - 2 s_0 V_x(p_0). 
    \]
    By the above discussion these two vectors together with $V_y(p_0)$ are linearly independent. This shows that $\{ V_y(p_0) \}$ is a generating set for $T_{p_0} (X_\beta \backslash H_c)$.  

    (2) The same arguments in (1) show the weak density property of $X_\beta$ tangent to $H_c$. 

    (3) It remains to show the Property (PO) for $(X_\beta, H_c)$. For $\lambda \ge 0$ consider the scaling automorphism of $X_\beta$
    \begin{align*}
        (x,y,s,t) \mapsto (e^{2\beta \lambda} x, e^{-2\beta \lambda} y, e^{\beta \lambda} s, e^{2 \lambda} t)
    \end{align*}
    mapping $H_c$ onto $H_{e^{2 \lambda}c}$, which leaves any compact subset $K \subseteq X_\beta$ for large $\lambda$. 
\end{proof}

\subsection{A topological observation}
It seems much more difficult to find examples satisfying the conditions of Theorem \ref{thm: typeII}. Not only do we need the weak density property of $(X, H)$, but also Property (PO). There could be topological obstructions to Property (PO), which compete with the density property of the complement $X \backslash H$. 

Consider the following example:
\begin{align*}
    X = SL_2(\C) = \left\{ \begin{pmatrix}
        a & b \\ c & d
    \end{pmatrix}:\ ad-bc=1 \right\}, 
    \quad 
    H=\{x\in X:\ a=0\}.
\end{align*} 
Then, the complement $SL_2(\C) \backslash H \cong \C_a^* \times \C_{b,c}^2$, which has the density property. Furthermore, one can verify the weak density property of $(SL_2(\C), H)$. However, there is  a topological obstruction to Property (PO) for $(SL_2(\C), H)$. Suppose that an automorphism $\alpha$ moves $H$ off $SU_2 \subseteq SL_2(\C)$, then the inclusions of topological spaces
\[
    SU_2 \overset{i_1}{\hookrightarrow} SL_2(\C) \backslash \alpha(H) \overset{i_2}{\hookrightarrow} SL_2(\C)
\]
induce homomorphisms of third homotopy groups
\[
    \pi_3(SU_2)  \overset{(i_{1})_*}{\longrightarrow} \pi_3(SL_2(\C) \backslash \alpha(H))  \overset{(i_{2})_*}{\longrightarrow} \pi_3(SL_2(\C)).
\]
Since $SU_2$ is diffeomorphic to $S^3$, $SL_2(\C)$ deformation retracts to $SU_2$ and the complement $SL_2(\C) \backslash \alpha(H)$ is biholomorphic to $ SL_2(\C) \backslash H$, we get 
\[
    \Z \overset{(i_{1})_*}{\longrightarrow}  0 \overset{(i_{2})_*}{\longrightarrow}  \Z.
\]
This contradicts the fact that $(i_2 \circ i_1)_* = Id$.

\medskip 

For another choice of hypersurface for $SL_2(\C)$, we have not spotted any topological obstruction to Property (PO). In addition, we can prove the other ingredients of Theorem \ref{thm: typeII}. Therefore, we conclude with a natural question. 

\begin{problem}\label{SL2_H}
    Let $H=\{ A \in SL_n(\C):\ \tr A=0\}$. Does there exist a Fatou-Bieberbach domain of the second kind (biholomorphic to $SL_n(\C)$) in $SL_n(\C) \backslash H$?
\end{problem}

\section{Application to the construction of non-tame families of Stein structures}
\label{sec: wild}

In their paper \cite{MR5101489}, Forstneri\v c and Sigur\dh ard\'{o}ttir introduced \textbf{tame families of Stein structures}. These are continuous families of Stein structures on a given real $2n$-dimensional manifold $M^{2n}$ such that a natural important generalization of the Oka principle holds. Their main theorem is as follows.

\begin{theorem}
Assume that $T$ is a finite CW complex, $X$ is a smooth manifold,
$J=\{J_t\}_{t\in T}$ is a tame family of smooth Stein structures on $X$
depending continuously on $t$, and $Y$ is an Oka manifold. Then every
continuous map $f: Z = T\times X \to Y$ is homotopic to a $J$-holomorphic map $F: Z\to Y$, that is, $F(t,\,\cdot\,) : X \to Y$ is $J_t$-holomorphic for every $t\in T$ and continuous in $t$. If $f$ is $J$-holomorphic on a neighborhood of a closed subset $K \subseteq Z$ with proper projection $K\to T$ and $J_t$-convex fibers $K_t$ $(t\in T)$, then $F$ can be chosen to approximate $f$ in the fine topology on $K$.
\end{theorem}

Here, the tameness is defined as follows.

\begin{definition}\label{def:tame}
A family $J=\{J_t\}_{t\in T}$ of complex structures on $X$ is \textbf{tame at a point} $t_0\in T$ if for every compact set $K \subseteq X$ and open set $U\subseteq X$ containing $\widehat K_{J_{t_0}}$ there is a neighborhood $T_0\subseteq T$ of $t_0$ such that $\widehat K_{J_t}\subseteq U$ holds for all $t\in T_0$. The family $J$ is \textbf{tame} if it is tame at every point $t_0\in T$.
\end{definition}

In the same paper, they show that without the assumption of tameness of the family of smooth Stein structures the Oka principle fails. Also, they give a construction of non-tame Stein structures on $\R^{2n}$, $n\ge 2$. In Remark~4.3 in that paper, they show that non-tame complex structures exist on bounded convex domains in $\C^n$, $n\ge 2$, and they wonder whether there exists on every Stein manifold of dimension at least $2$ a non-tame family.

As an application of our Theorems \ref{thm: typeII} and \ref{Hc}, we construct a non-tame family on $SL_n (\C) \cong SU_n \times \R^{n^2-1}$ for any $n \ge 2$. We begin with some preliminaries.

Let $K$ be a compact Lie group and $G = K^\C$ its complexification. Then, the polar decomposition gives a diffeomorphism: 
$$ 
    K\times \mathfrak{k}\rightarrow G,\ (k,\xi)\mapsto k\cdot \exp(i\xi).
$$
For $K = SU_n$, $G=SL_n (\C)$, this is the unique polar decomposition 
$g = k \cdot \exp(P)$ for any $g \in SL_n(\C)$ with $P$ Hermitian of trace $0$ and $k \in SU_n$.

Let $f : \mathfrak{k}\to\mathbb{R}$ be an $\mathrm{Ad}(K)$-invariant smooth function and set
\[
  \rho\bigl(k\cdot\exp(i\xi)\cdot k'\bigr) := f(\xi), \qquad k,k'\in K,\ \xi\in\mathfrak{k}.
\]
This function $\rho$ is well defined and $K\times K$-invariant. 

%The well known generalization of the Bochner tube theorem states that the function $\rho$ is plurisubharmonic if and only if $f$ is convex, and $\rho$ is strictly plurisubharmonic if $f$ is strictly convex. 

For our concrete example $SL_n (\C)$, the restriction of norm square from $\mathrm{Mat} (n, \C) \cong \C^{n^2}$ gives $\rho : SL_n (\C) \to \R$, $g \mapsto \mathrm{tr} ( g^*g)$. The corresponding convex function $f$ on the space of Hermitian matrices of trace zero ($\cong \R^{n^2-1}$)  given by 
$$
    A \mapsto \sum_{i=1}^n \exp (2 a_i), \text{ with } a_1, a_2, \dots, a_n \text{ the eigenvalues of } A,
$$ 
is a strictly convex function. The sublevel sets of $f$ are thus diffeomorphic to balls in $\R^{n^2-1}$. We denote these sublevel sets therefore by $\mathbb{B}_r$, even thought they are not balls in the standard metric of $\R^{n^2-1}$. For convenience, we replace $\rho$ by $\rho-n$, so that the minimum of $\rho$ becomes $0$, which is attained exactly at points of $SU_n$. For our purpose, it is only important to note that the (sublevel) sets $SU_n \times \mathbb{B}_r \subseteq SU_n \times \R^{n^2-1}$ are holomorphically convex in $SL_n (\C)$ as sublevel sets of $\rho$.

Consider for $t \in \R$ the map $D_t : SU_n \times \R^{n^2-1} \to SU_n \times \R^{n^2-1}$ given by $(k,v) \mapsto (k, t v)$. 

\begin{theorem}\label{thm:unbounded}
Given a compact set $L \subseteq SU_n \times \R^{n^2-1} $ $(n \ge 2)$ with $SU_n \times \{0\} \subseteq \operatorname{int} L$, there is a family of integrable smooth complex structures $\{J_t\}_{t\in\mathbb{R}}$ on $SU_n \times \R^{n^2-1}$, depending smoothly on $t\in\mathbb{R}$, such that $J_0$ is the standard complex structure on $SL_n (\C)$, $(SU_n \times \R^{n^2-1} ,J_t)$ is biholomorphic to
$(SU_n \times \R^{n^2-1} ,J_0)\cong SL_n (\C)$ for every $t\in\mathbb{R}$, and for any neighborhood $U\subseteq \mathbb{R}$ of $0\in\mathbb{R}$ the set
$\bigcup_{t\in U}  \widehat L_{J_t}\subseteq  SU_n \times \R^{n^2-1} $ is unbounded.
\end{theorem}
\begin{proof}
We follow closely the proof in \cite{MR5101489}. Assume without loss of generality that $L = SU_n \times \mathbb{B}_1$ and take $H= \{ a_{11} = c\}$    for $\vert c\vert $ sufficiently big such that $L \cap \alpha_s(H) = \emptyset$ for $s \ge 0$ with $\alpha_s$ the multiplication by $\diag(e^s,e^{-s},1,\dots,1)$. By Theorems \ref{thm: typeII} and \ref{Hc}, there exists an injective holomorphic map $\Phi: SL_n (\C) \to SL_n(\C)$ such that $\Phi (SL_n(\C)) \subseteq  SL_n (\C) \backslash H$, but $\widehat{\Phi (L)}$ contains a point $h \in H$ and
$$\Vert \Phi^{-1}- Id \Vert_{SU_n \times \mathbb{B}_{1/2}  }$$
is arbitrarily small. This implies by Rouch{\' e}'s theorem that on a slightly smaller compact, say $SU_n \times \mathbb{B}_{1/3}$, $\Phi$ is arbitrarily close to the identity. 

We shall construct a smooth family of diffeomorphisms 
$$
    \Psi_t : SL_n(\C) \to \Psi_t (SL_n(\C)) \subseteq SL_n (\C), \,\, t \in \R, 
$$
such that for every $t\ne 0$ we have $\Psi_t (SL_n (\C)) = SL_n (\C)$ and $\Psi_t = \Phi$ holds on a neighborhood of $D_{1/t} (L)$, while $\Psi_0 =\Phi$ on $SL_n (\C)$. When we achieve this, the proof of the assertion of the theorem follows exactly as in the proof in \cite{MR5101489}.
 
In order to construct the family of diffeomorphisms $\Psi_t$, we need a smooth isotopy of smooth embeddings $\Phi_s : SL_n(\C) \to SL_n(\C)$ (diffeomorphisms onto their image) $s\in [0,1]$ such that $\Phi_0 = Id$ on $SL_n(\C)$ and $\Phi_1 = \Phi$. After having constructed this family, we can again follow the proof in \cite{MR5101489} on p.~485 and deduce the existence of the family $\Psi_t$ above. The family $\Phi_s$ in \cite{MR5101489} is given by a concrete formula using the holomorphic retractibility of $\C^n$. We will use a similar idea but only in the radial direction. 

Write 
\[
    \Phi ( k, v ) = ( \alpha (k,v), \beta(k,v) ) , \quad k \in SU_n, \, v \in \R^{n^2-1},
\]
where 
\[
    \alpha : SU_n \times \R^{n^2-1} \to SU_n, \quad \beta : SU_n \times \R^{n^2-1} \to \R^{n^2-1}
\]
are real analytic maps. 
Since $\| \Phi - Id \|_{SU_n \times \mathbb{B}_{1/4} }$ is small, the Cauchy estimates yield that $\Phi - Id$ is $C^1$-small on $SU_n \times \{0\}$ and thus $k \mapsto \alpha(k,0)$ is a diffeomorphism of $SU_n$ (close to identity).

For $t \in (0,1]$, we consider 
\begin{align} \label{isotopy-v}
     ( k, v) \mapsto \left( \alpha(k,v),   \beta(k,v ) - (1-t^2) \beta( \alpha^{-1}( \alpha(k,v),0),0)  \right). 
\end{align}

For injectivity we remark that 
this is the composition of $\Phi = (\alpha, \beta)$ with the ``shear'' diffeomorphisms
$$
    S_t: (k,v) \mapsto (k, v -  (1-t^2) f(k)), \,\, \mathrm{ where  } \,\,  f(k) := 
\beta( \alpha^{-1}( k,0),0  ).
$$
%The second component of \eqref{isotopy-v} vanishes at $t=0, v=0$ for all $ k \in SU_n$. 
Now further for $t\in (0,1]$ consider the conjugation $D_{1/t} \circ S_t \circ \Phi \circ D_t$ with dilation $D_t$.
By real analyticity, this converges for $t \to 0$ to
\begin{align*}
    (k,v) \mapsto \left( \alpha(k,0), \, A(k) v \right)
\end{align*}
where $A : SU_n \to GL_{n^2-1}(\R)$ is given by
\[
    k \mapsto \partial_v \beta(k,0) - D f( \alpha(k,0)) \cdot \partial_v \alpha(k,0) .
\]
Since $\Phi - Id$ is $C^1$-small on $SU_n \times \{0\}$,  $\partial_v \beta(k,0) - I_{n^2-1}$ and $\partial_v \alpha(k,0)$ are small. Hence $A(SU_n)$ lies in a small neighborhood $V$ of the identity matrix $I_{n^2-1} \in GL_{n^2-1}(\R)$. 
By deformation retracting such a neighborhood $V$ to $I_{n^2-1}$ by a map $[0,1] \times V \to V$ we connect $\Phi_0$ further to the map
\[
    (k,v) \mapsto ( \alpha(k,0), v).
\]
% Now connect via $(s,k,v) \mapsto (\alpha(k, s v), v)$ to the map 
% \[
%     (k,v) \mapsto (\alpha (k,0), v)
% \]
% where $\alpha(\cdot,0)$ is $C^1$-close to the identity diffeomorphism of $SU_n$. 
Finally, $k \mapsto \alpha(k,0) k^{-1}$ is $C^1$-close to the constant identity map on $SU_n$
thus there exists $\gamma : SU_n \to \mathfrak{su}_n$ such that 
\[
    \alpha(k,0) k^{-1} = \exp(\gamma(k)) .
\]
Then the homotopy 
\begin{align*} 
    (t, k) \mapsto  \exp(t \gamma(k)) \, k \in SU_n 
\end{align*}
is a path of diffeomorphisms of $SU_n$, since diffeomorphisms of compact manifolds are open with respect to the $C^1$-topology. We have connected $(\alpha(k,0), v)$ to $(k,v)$. 

By reparametrizing this piecewise smooth (real analytic) path from $\Phi$ to $Id_{SL_n(\C)}$ we get a smooth path from $\Phi$ to $Id$ on $SL_n(\C)$. 
\end{proof}

\begin{remark}
    The Koras-Russell threefolds with odd $\beta \ge 3$ from Theorem~\ref{cubic} are known to be diffeomorphic to $\R^6$, in particular contractible. Therefore the existence of non-tame families of Stein structures on them follows as in \cites{MR5101489, FL26}. However, it is open whether these threefolds are biholomorphic to $\C^3$. If not, they give new non-tame Stein structures on $\R^6$.
\end{remark}

\section*{Acknowledgment}
We would like to thank Franc Forstneri{\v c} for pointing out the connection to the non-tame Stein structures.

\end{document}